\documentclass[11pt]{amsart}
\usepackage[utf8]{inputenc}
\usepackage[T1]{fontenc}
\usepackage{amsmath}
\usepackage{amsfonts}
\usepackage{amssymb}
\usepackage{amsthm}
\usepackage{mathtools}
\usepackage[all,cmtip]{xy}
\usepackage{mathabx}
\usepackage{color}
\definecolor{cobalt}{RGB}{61,89,171}
\usepackage[colorlinks,citecolor=cobalt,linkcolor=cobalt,urlcolor=cobalt,pdfpagemode=UseNone,backref = page]{hyperref}

\def\bN{\mathbb{N}}
\def\Q{\mathbb{Q}}
\def\QQ{\mathbb{Q}}
\def\bQ{\mathbb{Q}}
\def\R{\mathbb{R}}
\def\bR{\mathbb{R}}

\def\Z{\mathbb{Z}}
\def\bZ{\mathbb{Z}}

\def\bk{\mathbf{k}}

\def\fg{\mathfrak{g}}
\def\fn{\mathfrak{n}}

\def\cat{\mathrm{cat}}
\def\cd{\mathrm{cd}}

\theoremstyle{plain}
\newtheorem{theorem}{Theorem}[section]
\newtheorem{proposition}[theorem]{Proposition}
\newtheorem{lemma}[theorem]{Lemma}
\newtheorem{corollary}[theorem]{Corollary}
\newtheorem{conjecture}[theorem]{Conjecture}

\newtheorem{introtheorem}{Theorem}

\theoremstyle{definition}
\newtheorem{definition}[theorem]{Definition}
\newtheorem{example}[theorem]{Example}

\theoremstyle{remark}
\newtheorem{remark}[theorem]{Remark}

\newcommand{\thmref}[1]{Theorem~\ref{#1}}
\newcommand{\propref}[1]{Proposition~\ref{#1}}
\newcommand{\lemref}[1]{Lemma~\ref{#1}}
\newcommand{\corref}[1]{Corollary~\ref{#1}}

\newcommand{\remref}[1]{Remark~\ref{#1}}
\newcommand{\examref}[1]{Example~\ref{#1}}
\newcommand{\defref}[1]{Definition~\ref{#1}}
\newcommand{\subsecref}[1]{Subsection~\ref{#1}}

\def\CP{\mathbb C{\rm P}}

\def\Hom{\mathrm{Hom}}

\def\e_0{{\sf {e_0}}}
\def\cat{{\sf {cat}}}
\def\dim{{\sf {dim}}}
\def\b_1{{\sf {b_1}}}
\def\om{\omega}
\def\cupl{{\sf {cup}}}

\begin{document}

\title[Homotopy Invariants and Almost Non-Negative Curvature] {Homotopy Invariants and Almost Non-Negative Curvature}

\author[G. Bazzoni]{Giovanni Bazzoni}
\address{Departamento de \'Algebra, Geometr\'ia y Topolog\'ia, Facultad de Ciencias Matem\'aticas, Universidad Complutense de Madrid, 28040 Madrid, Spain}
\email{gbazzoni@ucm.es}

\author[G. Lupton]{Gregory Lupton}
\address{Department of Mathematics\\
Cleveland State University\\
Cleveland OH \\
44115 USA}
\email{g.lupton@csuohio.edu}

\author[J. Oprea]{John Oprea}
\address{Department of Mathematics\\
Cleveland State University\\
Cleveland OH \\
44115 USA}
\email{j.oprea@csuohio.edu}

\date{\today}


\begin{abstract} This paper explores the relation between the structure of fibre bundles akin to those associated to a closed almost nonnegatively sectionally curved manifold and rational homotopy theory.
\end{abstract}

\maketitle

%%%%%%%%%%%%%%%%%%%%%%%%%%%%%%%%%%%%%%%%%%%%%%%%%%%%%%%%%%%%%%%%%%%%%%%%%%%%%%%%%%%%%
%                                                                                   %
%                              INTRODUCTION                                         %
%                                                                                   %
%%%%%%%%%%%%%%%%%%%%%%%%%%%%%%%%%%%%%%%%%%%%%%%%%%%%%%%%%%%%%%%%%%%%%%%%%%%%%%%%%%%%%

\section{Introduction}\label{sec:intro}
A closed smooth manifold $M^m$ is said to be \emph{almost non-negatively (sectionally) curved} (or \emph{ANSC}) if it admits a sequence of
Riemannian metrics $\{g_n\}_{n\in \mathrm N}$ whose sectional curvatures $K_{g_n}$ and diameters ${\rm diam}_{g_n}$ satisfy
\[
 K_{g_n} \geq -\frac{1}{n}\ \ {\rm and}\ \ {\rm diam}_{g_n} \leq \frac{1}{n}.
\]
This is equivalent to the more common definition of saying that, for each $\epsilon > 0$, there exists a metric $g$ such that $K_g \cdot {\rm diam}_g^2 \geq -\epsilon$. ANSC manifolds generalize almost flat manifolds as well as manifolds with non-negative sectional curvature. Recently, in \cite{KPT} two remarkable theorems were proved that link curvature with intrinsic homotopy structure.

\begin{introtheorem}{\cite[Theorem A]{KPT}}\label{thm:nilpfincov}
 A closed ANSC manifold $M$ has a finite cover $\widehat M$ that is a nilpotent space in the sense of homotopy
theory.
\end{introtheorem}

\begin{introtheorem}{\cite[Theorem C]{KPT}}\label{thm:kpt}
If $M$ is a closed ANSC manifold, then there is a finite cover $\overline M$ that is the total space of a fiber bundle
\[
F \to \overline M \stackrel{p}{\to} N\,,
\]
where $N=K(\pi,1)$ is a nilmanifold and $F$ is a simply connected closed manifold.
\end{introtheorem}

\begin{remark}
In fact, the fibre $F$ is almost non-negatively curved in a certain generalized sense. Because we will not deal with this property, we refer the interested reader to \cite{KPT} for the precise definition.
\end{remark}

\begin{remark}\label{rem:equivfib}
Because $\pi_1(F)=0$ and $N=K(\pi,1)$, the bundle $F \to \overline M \to N$ is homotopy equivalent to the classifying fibration for the universal cover, $\widetilde M \to \overline M \stackrel{j_1}{\to} K(\pi,1)$. (Here, note that $\widetilde M$ is the universal cover of $M$ as well as of $\overline M$.) This means that hypotheses we make about $F$ below can equally well be viewed as hypotheses on the universal cover $\widetilde M$.

Of course, $\pi$ is an infinite (in fact, torsionfree nilpotent) group, so $\widetilde M$ is non-compact. Therefore, it seems strange on the face of it that we have
$\widetilde M \simeq F$ with $F$ compact, but in fact, this is not so unusual. For instance, the universal cover of $S^2 \times S^1$ is $S^2 \times \bR$
while the fiber of $S^2 \times S^1 \to S^1$ is the compact manifold $S^2$ of the same homotopy type as $S^2 \times \bR$.
\end{remark}

This paper began as a general exploration of the homotopical properties of bundles of ANSC-type $F \to \overline M \to N$, but evolved into a more
focused attempt to understand the interrelationships between \emph{nilpotency}, compactness of the fibre (with consequences and generalizations) and the rational
homotopy structure of the bundle. Because we wish to use the tools of rational homotopy to investigate the structure of ANSC manifolds, we need to know
that the bundle in Theorem \ref{thm:kpt} is of a particular type called \emph{quasi-nilpotent}. Indeed, the authors of \cite{KPT} never state that the space
$\overline M$ is itself nilpotent and this is essential to apply rational homotopy theory here. Thus, our first order of business is to place Theorem \ref{thm:kpt}
inside the rational world.

\section{The First Reduction}\label{sec:firstred}
To begin, let's recall the notion of nilpotent space.

\subsection{Nilpotence}\label{subsec:nilp}
\begin{definition}\label{def:nilp}
A space $X$ is \emph{nilpotent} if $\pi_1(X)$ is a nilpotent group and the standard action of $\pi_1(X)$ on each $\pi_j(X)$ for $j \geq 2$ is a nilpotent
action.
\end{definition}

This requirement is equivalent to saying that $\pi_1(X)$ acts nilpotently on $H_*(\widetilde X)$ where $\widetilde X$ is the universal cover of $X$, see \cite{HMR}.

\begin{remark}
Let a group $G$ act on a group $H$ via $g \cdot h$. Form
\[
\Gamma_2(H) = \{(g\cdot h)h^{-1} \mid g \in G,\ h \in H\}
\]
and let $\Gamma_n(H) = \Gamma_2(\Gamma_{n-1}(H))$. Then we have a sequence
\[
\Gamma_1(H)\coloneqq H \supseteq \Gamma_2(H) \supseteq \ldots \supseteq \Gamma_k(H) \supseteq \ldots\ .
\]
If for some $k$, $\Gamma_k(H)=\{e\}$, then the $G$-action on $H$ is said to be nilpotent.
\end{remark}

\begin{remark}
If $H$ is a group acting on itself by conjugation, then the action is nilpotent exactly when $H$ is a nilpotent group. Note that the $\Gamma_i(H)$'s in this case form the lower central series of $H$.
\end{remark}

\begin{remark}
A nilpotent action on a vector space $V$ is what is usually called a \emph{unipotent} action. That is, there is a finite sequence of subspaces
\[
V_k=\{0\} \subset V_{k-1} \subset \cdots V_1 \subset V_0=V\,.
\]
such that the action on each quotient $V_j/V_{j+1}$ is trivial. This is relevant for rational homotopy since the ``fundamental group'' acts on the
rational vector spaces $\pi_j(X) \otimes \bQ$.
\end{remark}

Recall that a nilmanifold $N$ is a compact quotient of a nilpotent Lie group by a finitely generated torsionfree discrete subgroup: $N = G/\pi$. Since nilpotent Lie groups are diffeomorphic to Euclidean spaces, we see that $N=K(\pi,1)$, so $N$ is a nilpotent space.

Now suppose $M$ is ANSC with fundamental group $\pi_1(M)$. From Theorem \ref{thm:nilpfincov}, because $\widehat M$ is nilpotent and a finite cover of $M$, we see that $\pi_1(\widehat M) = \widehat\Gamma$ is a finite index nilpotent subgroup of $\pi_1(M)$ which acts nilpotently on $H_*(\widetilde M)$. From Theorem \ref{thm:kpt}, since $F$ is simply connected, $\overline M$ is a finite cover and $N$ is a nilmanifold, we see that $\pi_1(\overline M) = \pi_1(N)=\overline \Gamma$ is a torsionfree nilpotent group of finite index in $\pi_1(M)$. Let
\[
\Gamma = \widehat\Gamma \cap \overline\Gamma
\]
and note that $\Gamma$ is torsionfree and acts nilpotently on $H_*(\widetilde M) \cong H_*(F)$ (by Remark \ref{rem:equivfib}). Furthermore, since the intersection of finite index subgroups has finite index as well, we see that the indices obey
\[
[\pi_1(M):\Gamma]<\infty, \quad [\widehat\Gamma:\Gamma]<\infty, \quad [\overline\Gamma:\Gamma]<\infty\,.
\]
Finally, since $\Gamma$ has finite index in $\overline\Gamma$, it is a lattice in the nilpotent Lie group defining $N=G/\overline\Gamma$ too. Then
the map $q\colon N'=G/\Gamma \to G/\overline\Gamma=N$ is a finite covering of order $[\overline \Gamma:\Gamma]$. Now consider the following pullback diagram.
\[
\xymatrix{
Q \ar[r]^-s \ar[d]_-r & N' \ar[d]^-q \\
\overline M \ar[r]^-p & N.
}
\]
Because $q$ is a cover and $p$ induces an isomorphism on fundamental groups, we see that $r$ is a connected finite cover of the same order as $q$. Because $p$ is a bundle map, we see that $s$ is a bundle map with the same fibre $F$. But now $Q$ has fundamental group $\Gamma$ that acts nilpotently on $H_*(F) \cong H_*(\widetilde M)$. Thus $Q$ is a nilpotent space. Therefore,
(returning to the original notation) Theorem \ref{thm:kpt} can be replaced by

\begin{theorem}\label{thm:replace}
If $M$ is a closed ANSC manifold, then there is an orientable finite cover $\overline M$ that is the total space of a fiber bundle
\[
 F \to \overline M \stackrel{p}{\to} N\,,
\]
where $\overline M$ is a nilpotent space, $N=K(\pi,1)$ is a nilmanifold and $F$ is a simply connected closed manifold (of generalized non-negative
sectional curvature).
\end{theorem}

\begin{remark}\label{rem:orientability}
 The orientability of the cover, which is not evident here, will be shown within the proof of \propref{prop:cat}.
\end{remark}

We will see that this elementary reduction allows for an interesting classical corollary in addition to later rational consequences.

\subsection{ANSC-type bundles}\label{subsec:anscbun}

The approach we take to understanding ANSC-manifolds is a blend of classical homotopy techniques and more modern rational homotopy methods. However, in order to use rational homotopy theory, certain requirements must be satisfied. In fact, for a fibre bundle $F\to E\stackrel{p}{\to}B$ to admit nice models in the sense of rational homotopy theory, the map $p$ has to be \emph{quasi-nilpotent}, meaning that $\pi_1(B)$ acts nilpotently on (either $H_*(F;\bZ)$ or) $H^*(F;\bZ)$, and consequently on $H^*(F;\Q)$ or $H^*(F;\R)$. Before defining the main object of interest of this paper, let us make sure that we can apply rational homotopy theory to the kind of fibre bundle obtained from ANSC manifolds using \thmref{thm:replace}.

As we noted above, $N$ is a nilpotent space. It follows from \cite{KPT} that $\overline{M}$ is a nilpotent space as well. Therefore, by \cite[Chapter II, Proposition 2.13]{HMR}, $p$ is a nilpotent map\footnote{Recall that a map $p\colon E\to B$ is {\em nilpotent} if $\pi_1(E)$ acts nilpotently on $\pi_*(F)$ (here $F$ is the homotopy fiber of $p$. Recall that this action is given via the isomorphism $\pi_j(B,E) \cong \pi_{i-1}(F)$ for all $j$ and the usual action of $\pi_1(E)$ on the homotopy of the (mapping cylinder) pair $\pi_j(B,E)$.}. Now, it is a nontrivial fact proved in \cite[Chapter 2, 5.4]{BK} that $p$ is indeed quasi-nilpotent. Thus the following definition is ready to be investigated with rational homotopic methods.

\begin{definition}\label{def:ANSCfib}
Say that a fibre bundle $F \to \overline M \stackrel{p}{\to} N$ is of \emph{ANSC-type} if $F$ is a simply connected closed manifold, $\overline M$ is a nilpotent
manifold and $N$ is a nilmanifold.
\end{definition}

As anticipated, we will analyze ANSC-type bundles from the viewpoint of homotopy theory. In particular, we will use rational homotopy theory to derive a numerical relationship among the three constituent spaces $F$, $\overline M$ and $N$ that leads to interesting Bochner-type results for ANSC-manifolds.

As an appetizer, here is a classical consequence of Theorem \ref{thm:replace}. Recall that, for a smooth closed $4k$-manifold $M$, the signature $\sigma(M)$ is defined to be the signature of the symmetric bilinear form
\[
H^{2k}(M;\R) \times H^{2k}(M;\R) \to H^{4k}(M;\R) \cong \R
\]
given by Poincar\'e duality. The matrix of the form is symmetric and non-degenerate, so it can be diagonalized to a matrix with real eigenvalues. Then
$\sigma(M)$ is the number of positive eigenvalues minus the number of negative eigenvalues. The signature is defined to be zero unless the
manifold has dimension a multiple of $4$. It is a standard fact that $\sigma(M \times N) = \sigma(M)\cdot\sigma(N)$.
In \cite{Ro}, the following more general multiplicative property of signature $\sigma$ was proven.

\begin{theorem}\label{thm:roit}
If $F \to E \to B$ is a smooth quasi-nilpotent fibre bundle (over $\R$) of coherently oriented smooth closed manifolds, then
\[
\sigma(E) = \sigma(B)\cdot\sigma(F)\,.
\]
\end{theorem}

As noted in \cite{Ro}, this was known for bundles where the action was trivial and it was known \emph{not} to hold in general. As for many results in topology, nilpotency is an adequate substitute for simplicity (i.e. trivial action). With Theorem \ref{thm:roit} in mind, we have the following application
of Theorem \ref{thm:replace}.

\begin{corollary}\label{cor:signature}
If $M$ is a compact ANSC manifold with infinite fundamental group, then $\sigma(M)=0$.
\end{corollary}

\begin{proof}
By Theorem \ref{thm:replace}, $M$ has a finite cover $\overline M$ which is nilpotent and sits as the total space in a smooth bundle
$F \to \overline M \to N$ where $F$ is a simply connected closed manifold and $N$ is a nilmanifold. By the discussion above, this bundle is quasi-nilpotent, so $\sigma(\overline M)=\sigma(N)\cdot\sigma(F)$. But since $N$ is a nilmanifold, it is parallelizable. Hence its
Stiefel-Whitney classes and Pontryagin classes vanish. By a classical theorem of Wall, this means that $N$ bounds an orientable
manifold. Hence, its signature $\sigma(N)$ vanishes. Thus, $\sigma(\overline M)=0$. Now, $\overline M$ is a finite cover of $M$ (of order $k$ say), so $\sigma(\overline M) = k\,\sigma(M)$ showing that $\sigma(M)=0$ as well.
\end{proof}

Now, let's see how Theorem \ref{thm:replace} can be used rationally to constrain $M$. To begin, we need to recall some facts about rational homotopy theory.

%%%%%%%%%%%%%%%%%%%%%%%%%%%%%%%%%%%%%%%%%%%%%%%%%%%%%%%%%%%%%%%%%%%%%%%%%%%%%%%%%%%%%
%                                                                                   %
%                          RATIONAL HOMOTOPY STRUCTURE                              %
%                                                                                   %
%%%%%%%%%%%%%%%%%%%%%%%%%%%%%%%%%%%%%%%%%%%%%%%%%%%%%%%%%%%%%%%%%%%%%%%%%%%%%%%%%%%%%

\section{Rational Homotopy Structure}\label{sec:rht}
The reader is referred to \cite{FHT,FHT2}, \cite[Chapters 2 and 3]{FOT} for details and proofs of the statements that follow.

A commutative graded algebra (cga) over a field of characteristic zero $\bk$, $A$, is called
\emph{free graded commutative} if $A$ is the quotient of $TV$, the tensor algebra on the graded vector space $V$,
by the bilateral ideal generated by the elements $a\otimes b - (-1)^{\vert a\vert \cdot \vert b\vert}
b\otimes a$, where $a$ and $b$ are homogeneous elements of $A$. As an algebra, $A$ is the tensor
product of the symmetric algebra on $V^{\mbox{\scriptsize even}}$ with the exterior algebra on
$V^{\mbox{\scriptsize odd}}$:
\[
A = \mbox{Symmetric}(V^{\mbox{\scriptsize even}})\otimes
\mbox{Exterior}(V^{\mbox{\scriptsize odd}})\,.
\]

\noindent We denote the free commutative graded algebra on the graded vector space $V$ by $\Lambda V$.
Note that this notation refers to a free commutative graded algebra and not necessarily to an exterior
algebra alone. We usually write $\Lambda V = \Lambda (x_i)$, where $x_i$ is a homogeneous basis of $V$.
Clearly the cohomology of a cdga is a commutative graded algebra. A morphism of cdga's inducing
an isomorphism in cohomology will be called a \emph{quasi-isomorphism}.
A \emph{Sullivan cdga} is a cdga $(\Lambda V,d)$ whose underlying
algebra is free commutative, with $V = \{\, V^n\,\}$, $n\geq 1$, and such that $V$ admits a basis
$x_\alpha$ indexed by a well-ordered set such that
\[
d(x_\alpha) \in\Lambda(x_\beta)_{\beta<\alpha}\,.
\]

\subsection{Minimal models}\label{subsec:min}
A \emph{(Sullivan) minimal cdga} is a Sullivan cdga $(\Lambda V,d)$ satisfying the additional property that
$d(V)\subset \Lambda^{\geq 2} V$. Minimal cdga's play an important role because they are tractable models for
``all'' other cdga's. (For the path-connected non-simply-connected case of the following result, see
\cite[Chapter 6]{Hal}.)

\begin{theorem}[Existence and Uniqueness of the Minimal Model]\label{thm:existuniqminmod} \hfill\newline
Let $(A,d)$ be a cdga over $\bk$ satisfying $H^0(A,d) = \bk$,
where $\bk$ is $\mathbb R$ or $\mathbb Q$ and ${\rm dim}(H^p(A,d))<\infty$
for all $p$. Then,
\begin{enumerate}
\item There is a quasi-isomorphism $\varphi \colon (\Lambda V,d) \to (A,d)$,
where $(\Lambda V,d)$ is a minimal cdga.
\item The minimal cdga $(\Lambda V,d)$ is unique in the following sense:
If $(\Lambda W,d)$ is a minimal cdga and $\psi\colon (\Lambda W,d) \to
(A,d)$ is a quasi-isomorphism, then there is an isomorphism
$f\colon (\Lambda  V,d) \to (\Lambda W,d)$ such that $\psi\circ f$ is
homotopic (see \cite{FHT}) to $\varphi$.
\end{enumerate}

\noindent The cdga $(\Lambda V,d)$ is then called the \emph{minimal model} of $(A,d)$.
\end{theorem}

The connection between this type of algebra and topology is via the de Rham cdga of differential forms on the
manifold $M$, $(\Omega(M),d)$, when $\bk$ is $\mathbb R$ and Sullivan's rational polynomial forms on $M$,
$(A_{PL}(M),d)$, when $\bk$ is $\mathbb Q$. Note that we have the de Rham theorems:
\[
H^*(\Omega(M),d) \cong H^*(M;\R) \quad {\rm and} \quad H^*(A_{PL}(M),d) \cong H^*(M;\Q)
\]
where the right side of each isomorphism denotes singular cohomology. Applying \thmref{thm:existuniqminmod} to these cdga's produces a \emph{minimal model of the space} $M$ denoted by
$\varphi\colon \mathcal M_M = (\Lambda V,d) \to A$, where we let $A$ stand for either the de Rham or Sullivan
algebras. We shall not distinguish the minimal models depending on the field because the context will always be clear. The minimal model thus provides a special type of cdga associated to a space. Note that the condition $H^0(A,d) = \bk$ in \thmref{thm:existuniqminmod} means that any path-connected space
has a minimal model (but the model may not accurately reflect homotopy properties of the space when the space
is not nilpotent). There are several key facts that make minimal cdga's an important tool. Say that the spaces $X$ and $Y$ have \emph{the same rational homotopy type}, denoted $X\simeq_\bQ Y$, if there is a finite chain of maps $X \to Y_1 \leftarrow Y_2 \to \cdots \to Y$ such that each induced map in rational cohomology is an isomorphism.

\begin{proposition}\label{prop:rhtype}
If $X$ and $Y$ have the same rational homotopy type, then their minimal models are isomorphic.
Moreover, if $X$ and $Y$ are nilpotent spaces (e.g. simply connected), then the converse is true.
\end{proposition}
The second statement follows from the existence of spatial rationalizations $X_\Q$ coming from homotopical
localization theory. In general, these do not exist for non-nilpotent spaces.
Except for the existence of a localization, everything we have said applies to models over $\mathbb R$ as well.

\begin{definition}\label{def:formal}
A space $X$, with minimal model $(\Lambda V,d)$, is called {\em formal} if there is a quasi-isomorphism
\[
\theta \colon (\Lambda V,d) \to (H^*(X;\bQ),0)\,.
\]
\end{definition}

\begin{example}\label{exam:tori}
Let $T^n$ denote the $n$-torus. The cohomology $H^*(T^n;\Q)=\Lambda(x_1,\ldots,x_n)$ is an exterior algebra on $n$ generators
in degree one, so is free as an algebra. Then, denoting Sullivan or de Rham forms by $A$, we can define a cdga homomorphism
$\varphi\colon H^*(T^n;\Q) \to A$ by simply assigning $x_j$ to any cocycle in $A$ representing 
$x_j\in H^*(A,d) \cong H^*(T^n;\Q)$. Because $\varphi$ induces an isomorphism on cohomology and $H^*(T^n;\Q)$ is free as an algebra,
we see that $H^*(T^n;\Q)$ itself is the minimal model of $T^n$.
\end{example}

\begin{example}[See \cite{FOT}]\label{exam:nilman}
To any nilmanifold $N=K(\pi,1)$, we can associate a rational nilpotent Lie algebra
$\fg$ with the property that there exists a basis in
$\fg$, $\{X_1,\ldots,X_n\}$, such that the structure
constants $\{c_{ij}^k\}$ arising in brackets
\begin{equation}\label{eq:brackets}
[X_i,X_j] = \sum_k c_{ij}^k X_k
\end{equation}
are rational numbers for all $i,j,k$. In fact, corresponding to $\fg$, there is an $n$-dimensional, simply
connected nilpotent Lie group $G$ which admits a discrete co-compact
subgroup $\pi$ so that $N=G/\pi$ is a compact nilmanifold.

Let $\fg$ have basis $\{X_1,\ldots,X_n\}$; the dual of
$\fg$, $\fg^*$, has basis $\{v_1,\ldots,v_n\}$ and
there is a differential $d$ on the \emph{exterior algebra}
$\Lambda \fg^*$ given by defining it to be dual to the bracket on degree $1$ elements,
\[
(dv_k)(X_i,X_j)=-v_k([X_i,X_j])\,,
\]
and then extending $d$ to be a graded derivation.
Now using \eqref{eq:brackets}, duality gives
\[
(d v_k)(X_i,X_j) = - c_{ij}^k
\]
and the differential on generators has the form
\[
d v_k = - \sum_{i<j} c_{ij}^k v_i \wedge v_j\,.
\]
We note that the Jacobi identity in the Lie algebra
is equivalent to the condition $d^2=0$. Therefore, we obtain
a \emph{commutative differential graded algebra} (or cdga)
$(\Lambda \fg^* ,d)$ associated to the Lie algebra
$\fg$. The fundamental result here is the following.

\begin{theorem}\label{thm:nilminmod}
If $N=G/\pi$ is a nilmanifold, then the cdga $(\Lambda \fg^*,d)$ associated to
$\fg$ is a minimal model for $N$ and, thus, computes all of the rational homotopy information about $N$.
\end{theorem}

\begin{remark}
 In the theory of Lie algebras, the cdga $(\Lambda \fg^*,d)$ is known as {\em Chevalley-Eilenberg complex} of $\fg$.
\end{remark}

Now, the minimal model of $N$ has the form
\[
{\mathcal M}_N =(\Lambda(v_1,\ldots v_n),d) \qquad {\rm with}
\qquad |v_i|=1\,,
\]
where the nilpotency of $\fn$ converts by duality into
the condition that the differential on $v_j$ is a polynomial
in $v_k$ with $k < j$ having no linear terms. In fact, this can be refined to say that
the generators are added in stages and the generators in the
$j\textsuperscript{th}$ stage have differentials that are polynomials in the
generators of stages $1$ through $j-1$. In particular,
because $\fg$ is nilpotent, there is a non-trivial
complement to $[\fg,\fg] \subset \fg$
which is isomorphic to $\fg/[\fg,\fg]
\cong H^1(N;\Q)$. Duality then says that there is some $k$ with
$2 \leq k \leq k$ such that $dv_i=0$ for $i \leq k$.

The minimal model ${\mathcal M}_N$ is an exterior algebra
so, since ${\rm degree}(v_j)=1$ for $1 \leq j \leq n$,
the top degree of a non-zero element is $n$ and a vector space
generator is $v_1\cdot v_2\cdots v_n$. This element is
obviously a cocycle, so $H^n(N;\Q)=\Q$; thus, $N$ is orientable and
has dimension $n$. From the discussion above, we also see that
$\b_1(N) \leq \dim(N)$. In fact, by the discussion above and the following lemma, if
$\b_1(N)=\dim(N)$, then $N$ is diffeomorphic to a torus of rank $\b_1(N)$.
\end{example}

\begin{lemma}\label{lem:rhtnil}
If a nilmanifold $N$ has the rational homotopy type of a torus, then it is diffeomorphic to a torus.
\end{lemma}
\begin{proof}
The hypothesis says that the homotopical localization map has the form $\phi\colon N \to N_\Q \simeq_\Q T^k$. But
on the fundamental group level, the kernel of rationalization consists of torsion. Because $\pi_1(N)$ is torsion free,
$\phi_*$ is injective. But then $\pi_1(N)$ must be a finitely generated torsionfree \emph{abelian} group, hence
$N$ has the homotopy type of a torus of rank $\b_1(N)$. Mostow rigidity then says that $N$ is diffeomorphic
to such a torus.
\end{proof}

\begin{remark}
 It turns out that a nilmanifold is formal if and only if it has the same rational homotopy type of (hence it is diffeomorphic to) a torus.
\end{remark}

One important use of minimal models is that they allow the construction of new (rational homotopy) invariants or, in some cases, new descriptions of familiar invariants. One such invariant is the following, the Toomer invariant, which may be defined in terms of the Milnor-Moore spectral sequence classically, i.e., without using minimal models (see \cite{FHT}).

Let $X$ be a nilpotent space with minimal model $(\Lambda W,d)$ and denote by $\rho_s$ the projection
\[
\rho_s\colon \Lambda W \to \frac{\Lambda W}{\Lambda^{> s} W}
\]
where $\Lambda^{> s} W$ signifies all the words in generators $W$ of length greater than $s$.

\begin{definition}\label{def:e0}
The {\em Toomer invariant} $\e_0(X)$ is the largest $k$ such that the projection $\rho_{k-1}$
is \emph{not} injective on cohomology. Equivalently, $\e_0(X)$ is the smallest $k$ such that $\rho_k$ \emph{is} injective
on cohomology. Similarly, for a class $\tau\in H^*(X;\Q)$, $\e_0(\tau)$ is the largest $k$ such that $\rho_{k-1}^*(\tau)=0$ or the smallest $k$ such that $\rho_k^*(\tau) \not = 0$.
\end{definition}

Note that, for a space with finite dimensional rational cohomology, eventually the word length exceeds the top degree with non-zero
cohomology, so for some $k$, $\rho_k$ is injective on cohomology. We have the following result that allows us to compute $\e_0$ in many cases.

\begin{proposition}{\cite[Lemma 10.1]{FH}}\label{prop:tope0}
If $X$ obeys rational Poincar\'e duality in cohomology with top class $\tau$, then $\e_0(\tau)=\e_0(X)$. Moreover,
$${\sf{e_0}}(X)=\max\{k\,|\, \tau\ {\rm is\ represented\ by\ a\ cocycle\ in}\ \Lambda^{\geq k}W\}.$$
\end{proposition}

For some spaces $X$ -- notably when $X$ is a \emph{formal} space (\defref{def:formal}) -- we have $\e_0(X) =  \cupl_{\Q}(X)$, which is the \emph{rational cup-length} of $X$. Namely,
$\cupl_{\Q}(X)$ is the longest non-zero product of elements in $H^*(X;\Q)$. But generally, $\e_0(X)$ is a finer invariant than $\cupl_{\Q}(X)$, and the inequality
$\cupl_{\Q}(X) \leq \e_0(X)$ is often strict. This is the case for most nilmanifolds: the only way in which we can have $\cupl_{\Q}(N) = \dim(N)$, for $N$ a nilmanifold,
is if we have $\b_1(N) = \dim(N)$, and from \lemref{lem:rhtnil} and the discussion preceding it, this corresponds to the case in which $N$ is rationally (actually up to diffeomorphism) a torus. For $\e_0(N)$, however, we have the following result.

\begin{corollary}\label{cor:e0nil}
If $N$ is a nilmanifold, then $\e_0(N)=\dim(N)$.
\end{corollary}

\begin{proof}
The minimal model of $N=G/\pi$ is given by $(\Lambda \fg^*,d)$ with $\fg^*=\langle v_1,\ldots,v_n\rangle$ and each $v_j$ of degree one. Because $\Lambda \fg^*$ is an exterior algebra on the $n$ generators $v_j$, the longest word in
$\Lambda V$ is given by $c=v_1 v_2 \cdots v_n \in (\Lambda \fg^*)^n$. Automatically then, this element $c$ is a cocycle. It is not a coboundary because $N$ is an orientable closed manifold of dimension the rank of the nilpotent fundamental group; that is, $n$. Hence, $c$ uniquely represents the top class of $H^*(N;\Q)$. By \propref{prop:tope0}, we only need to know $\e_0([c])$ to determine $\e_0(N)$. If $k<n$, then clearly $\rho_k^*([c])=0$ since words of length $n >k$ are killed.
On the other hand, $\rho_n^*([c]) \not = 0$ since only words of length longer than $n$ are killed. Thus, $\e_0([c])=n$.
\end{proof}

There is one important class of spaces whose cohomology obeys Poincar\'e duality. These are the (rationally) \emph{elliptic spaces}.
A space $X$ is \emph{elliptic} if both $\dim(H^*(X;\Q)) < \infty$ and $\dim(\pi_*(X)\otimes \Q) < \infty$. Homogeneous spaces
are prime examples of elliptic spaces. It turns out that elliptic spaces only come in two forms: either we have Euler characteristic $\chi=0$, or we have $\chi > 0$ (see \cite[Ch.32]{FHT} for this and other facts about elliptic spaces). If $\chi(X) > 0$, then
$X$ is called an $F_0$-space and is constrained by the facts that the minimal model has the form
$$(\Lambda(V^{\rm even} \oplus V^{\rm odd}),d), \ {\rm with}\ d(V^{\rm even})=0, \ d(V^{\rm odd}) \subset \Lambda^{\geq 2}(V^{\rm even}),$$
and $H^*(X;\bQ)$ is given by a polynomial algebra on even degree generators modulo an ideal generated by a regular sequence.

\subsection{Relative models}\label{subsec:relmodel}
There is a relative notion of model for fibrations. The algebraic basis for constructing such models is
as follows.

 \begin{definition}\label{def:relminmod}
A \emph{relative minimal cdga} is a morphism of cdga's of the form
\[
i \colon (A,d_A) \to (A\otimes \Lambda V,d)\,,
\]
where $i(a)= a$, $d_{\vert A} = d_A$, $d(V)\subset (A^+\otimes \Lambda V)
\oplus \Lambda^{\geq 2} V$, and such that $V$ admits a basis
$(x_{\alpha})$ indexed by a well-ordered set such that
$d(x_{\alpha}) \in A \otimes (\Lambda (x_\beta))_{ \beta <\alpha }$.
\end{definition}

When $(A,d_A)$ is a Sullivan cdga, we have $(A,d_A) = (\Lambda Z,d)$. Clearly, a relative minimal cdga $(A\otimes \Lambda V,d)
= (\Lambda (Z\oplus V),d)$ is also a Sullivan cdga, but the cdga $(\Lambda (Z\oplus V),d)$ is not necessarily a minimal cdga,
even if $(\Lambda Z,d)$ is a minimal cdga.
Relative Sullivan cdga's are in some sense the generic models for morphisms of cdga's. We make the role of relative minimal models
precise in the following theorem (see \cite[Section 14]{FHT}).

\begin{theorem}[{\it Relative version of \thmref{thm:existuniqminmod}}]\label{thm:relminmod}
Let $f \colon (A,d)\to (B,d)$ be a morphism of cdga's. We then have a commutative diagram
\[
 \xymatrix{
A\ar[r]^f\ar[dr]_-i&B\\
& (A \otimes \Lambda V,d)\ar[u]_g }
\]
where $i$ is a relative minimal cdga and $g$ is a quasi-isomorphism. This property characterizes $(A\otimes \Lambda V,d)$ up to isomorphism.
\end{theorem}

Under the conditions of \thmref{thm:relminmod}, the map $i$ is called \emph{the relative minimal model} of $f$. Let's see now how
this applies to fibrations. Recall from \subsecref{subsec:anscbun} that a fibration $F \to E \to B$ is ($\Q$-)\emph{quasi-nilpotent} if $\pi_1(B)$ acts nilpotently on $H^*(F;\Q)$.

Now let $F\to E \stackrel{p}{\to} B$ be a $\Q$-quasi-nilpotent fibration. We form the following commutative diagram
\[
\xymatrix{
A_{PL}(B) \ar[r]^-p&A_{PL}(E)\ar[r]&A_{PL}(F)\\
(\Lambda V,d)\ar[u]^{\varphi}\ar[r]^-i& (\Lambda V\otimes \Lambda W,d)
\ar[u]^\psi\ar[r]^-\rho &  (\Lambda W,\bar{d}) \ar[u]_{\bar{\psi}}.
}
\]
Here the morphism $\varphi \colon (\Lambda V,d) \to A_{PL}(B)$ is the minimal model of $B$, $\psi$ is a quasi-isomorphism and
$(\Lambda V,d)\to (\Lambda V \otimes \Lambda W,d)$ is a relative minimal cdga. The cdga $(\Lambda W,\bar d)$ is the quotient cdga $(\Lambda V\otimes
\Lambda W,d)/ (\Lambda^+(V)\otimes\Lambda W)$ and the map $\rho$ is the quotient map. The map $\bar \psi$ is induced by the commutativity of
the left-hand square of the diagram.

\begin{theorem}{\cite[Theorem 15.3]{FHT}}\label{thm:fibmod}
Suppose $F\to E \stackrel{p}{\to} B$ is a $\Q$-quasi-nilpotent fibration. If $B$ and $F$ have finite Betti
numbers and $H^1(p)$ is injective, then the map $\bar \psi$ is a quasi-isomorphism, and the cdga $(\Lambda W,\bar d)$ is the minimal
model of the fibre $F$.
\end{theorem}

There is one important example to keep in mind.

\begin{example}\label{exam:nilrel}
Suppose $X$ is a nilpotent space. That is, $\pi_1(X)=\pi$ is a nilpotent group
and $\pi$ acts nilpotently on the higher homotopy $\pi_{\geq 2}(X)$. Consider the fibration $\widetilde X \to X \to K(\pi,1)$ that classifies the
universal cover $\widetilde X$. By the discussion in \subsecref{subsec:anscbun}, we see that this fibration is quasi-nilpotent. By \thmref{thm:fibmod}
(noting that the $H^1(p)$ hypothesis is satisfied because $\pi_1(X) = \pi$), there is a relative model of the form
\[
(\Lambda V,d_V) \to (\Lambda V\otimes \Lambda W,D) \to (\Lambda W, d_W)
\]
where $(\Lambda V,d_V)$ is a model for $K(\pi,1)$ and $(\Lambda W, d_W)$ is a model for $\widetilde X$. By the general form of relative models, the differential $D$ has the form (for $v \in V$, $\chi \in \Lambda W$ and $\{v_1,\ldots,v_k\}$ a basis for $V$)
\[
 D(v)=d_V(v)\ {\rm and}\ D(\chi)=d_W(\chi) + \sum v_i\,\theta_i(\chi) + \chi_2
\]
for $\chi\in\Lambda W$, with $\chi_2\in\Lambda^{\geq 2}V\otimes \Lambda W$. We may use this to define linear maps $\theta_i$ of $\Lambda W$ for each $i$. Notice that since $\widetilde X$ is simply connected, $D$ has no linear part and $(\Lambda V\otimes \Lambda W,D)$ is actually the minimal model of $X$. Now a standard calculation shows that each $\theta_i$ is a degree-zero derivation of $\Lambda W$ that satisfies $\theta_i\circ d_W=d_W\circ \theta_i$: Since $D$ satisfies the Leibniz rule, we may equate like terms in $D(\chi \chi') = D(\chi)\chi' + (-1)^{|\chi|} \chi D(\chi')$ to show that $\theta_i$ is a derivation; we may expand out $D\big(D(\chi)\big) =0$ and equate like terms to show that $\theta_i\circ d_W=d_W\circ \theta_i$. This is the fundamental structure of the ANSC-type models we shall consider below.
\end{example}

%%%%%%%%%%%%%%%%%%%%%%%%%%%%%%%%%%%%%%%%%%%%%%%%%%%%%%%%%%%%%%%%%%%%%%%%%%%%%%%%%%%%%
%                                                                                   %
%               SECOND REDUCTION: RATIONAL HOMOTOPY AND ANSC                        %
%                                                                                   %
%%%%%%%%%%%%%%%%%%%%%%%%%%%%%%%%%%%%%%%%%%%%%%%%%%%%%%%%%%%%%%%%%%%%%%%%%%%%%%%%%%%%%

\section{The Second Reduction: Rational Homotopy and ANSC}\label{sec:cat}
\subsection{Inequalities for ANSC bundles}\label{subsec:main}
We work in slightly greater generality than ANSC-bundles. Here we allow the fibre of a fibration to be a simply connected Poincar\'e duality space over $\Q$. Recall that
this means that $H^*(F;\Q)$ obeys Poincar\'e duality with respect to a top class $[\mu] \in H^{\dim\, F}(F;\Q)\cong \Q$. From the discussion in Subsection \ref{subsec:anscbun}, we see that ANSC bundles satisfy the hypotheses of \propref{prop:cat} below. This result is a particular case of a general result about rational category due to Jessup (see \cite[Theorem 9.6]{FHT2}, or \cite[Proposition 9.7]{FHT2}, for the non-simply-connected version). Since we only deal with our particular type of structure, we are able to give a direct proof, that avoids much of the technical background used for the proof of \cite[Theorem 9.6]{FHT2} and also draws out the connections between the ANSC structure and rational invariants.

\begin{proposition}\label{prop:cat}
Let $F\to E\xrightarrow{p} N$ be a quasi-nilpotent fibration with $F$ a simply connected Poincar\'e duality space and $N$ a
nilmanifold. Then $\e_0(E)\geq \e_0(F)+\dim(N)$.
\end{proposition}

\begin{proof}
Let $n$ be the dimension of $N$. Then the minimal model of $N$ is an exterior algebra $(\Lambda V,d_V)$ with $V$ an $n$-dimensional vector space generated in degree 1, $V=\langle v_1,\ldots,v_n\rangle$, and $d_V\colon V\to \Lambda^2 V$. Let $(\Lambda W,d_W)$ be the minimal model of the fiber. Since the fibration is quasi-nilpotent, there is a relative model for the projection $p$:
\[
(\Lambda V,d_V)\to (\Lambda V\otimes \Lambda W,D)\to (\Lambda W,d_W)\,,
\]
where $D=d_V$ on $\Lambda V\subset \Lambda V\otimes \Lambda W$ and we may write
\begin{equation}\label{differential_D}
D\chi= d_W\chi+\sum_{i=1}^nv_i\theta_i(\chi)+\chi_2\,,
\end{equation}
as in \examref{exam:nilrel}.
As pointed out there, each $\theta_i$ is a degree-zero derivation that satisfies $\theta_i\circ d_W=d_W\circ \theta_i$ and thus induces a derivation $\theta_i^*$ on $H^*(\Lambda W)=H^*(F;\bQ)$. Moreover, $\theta_i^*$ is a nilpotent derivation, since we are dealing with a KS-model. If $m$ is the cohomological dimension of $F$, it follows from Poincar\'e duality that $\dim H^m(F;\bQ)=1$, hence $H^m(\Lambda W)=\langle [\mu]\rangle$ for some cocycle $\mu\in(\Lambda W)^m$, and $\theta_i([\mu])=0$ by nilpotency.

Notice, further, that if we have $\e_0(F)=\e_0([\mu])=r$, then we may assume that the representative cocycle $\mu$ satisfies $\mu \in \Lambda^{\geq r} W$. For suppose $\mu = \mu_{r-1} + \mu_r + \mu_{r+1}$, where $\mu_{r-1} \in \Lambda^{\leq r-1}$, $\mu_{r} \in \Lambda^{r}$, and $\mu_{r+1} \in \Lambda^{\geq r+1}$. From \defref{def:e0}, $\e_0(F)=r$ is the smallest $k$ such that $\rho_k^*([\mu]) \not = 0$. Therefore, we have $\rho_{r-1}^*([\mu])  = 0$. Now $\rho_{r-1}(\mu)= \mu_{r-1}$, and it follows that in $\Lambda W$, we have some $\eta$ with $d_W(\eta) = \mu_{r-1} + \mu'_r$, where now $\mu'_r \in \Lambda^{\geq r} W$. So we may take the cocycle representative of $[\mu]$ to be $\mu - d_W\eta \in \Lambda^{\geq r} W$.

Now, consider the element $v_1\cdots v_n\mu\in (\Lambda V\otimes\Lambda W)^{n+m}$. It is a cocycle, hence defines a cohomology class $[v_1\cdots v_n\mu]\in H^{m+n}(E;\bQ)$. We shall prove that it is nonzero. By \propref{prop:tope0}, since $v_1\cdots v_n\mu\in\Lambda^{\geq n+r}(V\oplus W)$, this is enough to guarantee that $\e_0(E)\geq \e_0(F)+\dim(N)$. In the notation of \thmref{thm:replace} and \remref{rem:orientability}, note that we are proving here that the total space of the bundle $F\to\overline M\to N$ has a rational top class and, hence, is orientable.

Next, we proceed to show that $[v_1\cdots v_n\mu]\neq 0$. To do so, we first suppose that $v_1\cdots v_n\mu=D(\eta_n)$, with $\eta_n\in\Lambda^nV\otimes\Lambda W$; hence $\eta_n=v_1\cdots v_n\tau$ for some $\tau\in\Lambda W$. Since $v_1\cdots v_n$ represents the fundamental class of $N$ and $D$ coincides with $d_V$ on $\Lambda V$, $D(v_1\cdots v_n)=0$, hence, in view of \eqref{differential_D},
\[
D(\eta_n)=(-1)^nv_1\cdots v_n D(\tau)=(-1)^nv_1\cdots v_n(d_W\tau)\,.
\]
But this implies that $\mu=d_W\tau$, which contradicts the assumption that $[\mu]$ is a fundamental class in $H^m(\Lambda W)$.

Next, suppose $n\geq 2$ and $v_1\cdots v_n\mu=D(\eta_0+\cdots+\eta_{n-1}+\eta_n)$, with $\eta_i\in\Lambda^i V\otimes \Lambda W$. Notice that $D(\Lambda^i V\otimes \Lambda W)\subset\Lambda^{\geq i} V\otimes \Lambda W)$. We have
\[
D(\eta_0)=d_W\eta_0+\textrm{terms in }\Lambda^{\geq 1}V\otimes\Lambda W\,.
\]
Of all the terms coming from $D(\eta_0+\cdots+\eta_{n-1}+\eta_n)$, $d_W\eta_0$ is the only one in $\Lambda W$, hence it must be zero, since $D(\eta_0+\cdots+\eta_{n-1}+\eta_n)=v_1\cdots v_n\mu$. This means that $\eta_0$ is a $d_W$-cocycle of degree $n+m-1>m$, hence it must be exact, $\eta_0=d_W\alpha_0$; now, by \eqref{differential_D},
\[
D(\alpha_0)=d_W\alpha_0+\xi_1=\eta_0+\xi_1\,,
\]
where $\xi_1\in\Lambda^{\geq 1} V\otimes\Lambda W$; thus
\begin{align*}
D(\eta_0+\cdots+\eta_{n-1}+\eta_n)&=D(D(\alpha_0)-\xi_1)+D(\eta_1+\cdots+\eta_{n-1}+\eta_n)=\\
&=D(\eta_1'+\cdots+\eta_{n-1}'+\eta_n')\,,
\end{align*}
and we can get rid of the degree 0 term. We can play the same trick to remove the $\eta_i$ term as long as $n+m-1-i>m$, which holds for $i<n-1$. For $n \geq 2$, then, if $v_1\cdots v_n\mu$ is to be a boundary, then we must have $D(\eta_{n-1}+\eta_n)=v_1\cdots v_n\mu$. Of course, if $n=1$ and $v_1\mu$ is to be a boundary, then we must have $D(\eta_{0}+\eta_1)=v_1\mu$.

Thus, the last case to discuss is $D(\eta_{n-1}+\eta_n)=v_1\cdots v_n\mu$, for $n \geq 1$. We may write $\eta_{n-1}=\sum_{i=1}^n v_1\cdots \widehat{v_i}\cdots v_n\beta_i$, where $\widehat{v_i}$ means $v_i$ is omitted, and $\beta_i\in (\Lambda W)^m$. Also, write $\eta_n = v_1\cdots v_n\tau$. Because $d_V \colon V \to \Lambda^2V$, we have $d_V(v_1\cdots \widehat{v_i}\cdots v_n) = 0$. So we compute
\begin{align*}
D(\eta_{n-1} + \eta_n)&=(-1)^{n-1}\sum_{i=1}^n v_1\cdots \widehat{v_i}\cdots v_nD(\beta_i) + (-1)^n v_1\cdots v_nD(\tau)\\
&=(-1)^{n-1}\sum_{i=1}^n v_1\cdots \widehat{v_i}\cdots v_n(d_W(\beta_i)+v_i\theta_i(\beta_i))\\
&+ (-1)^n v_1\cdots v_n d_W(\tau)\,,
\end{align*}
and infer from this that $d_W\beta_i$ vanishes for $i=1,\ldots,n$. Each $\beta_i$ defines a cohomology class $[\beta_i]\in H^m(\Lambda W)=\langle [\mu]\rangle$. This implies that $\theta_i(\beta_i)=d_W\gamma_i$, since $\theta^*_i([\mu]) = 0$. Hence
\[
D(\eta_{n-1}+\eta_n)=(-1)^{n-1}v_1\cdots v_nd_W\left(\sum_{i=1}^n\,(-1)^{n-i}\gamma_i-\tau\right)\,,
\]
which would again imply the exactness of $\mu$. So $v_1\cdots v_n\mu$ is not a boundary, and the result follows.
\end{proof}

\subsection{Geometric consequences: Bochner-like theorems}\label{subsec:bochner}
In this subsection, we will make use of the notion of Lusternik-Schnirelmann category for a space $X$. LS category $\cat(X)$ is a homotopy invariant defined as
the least integer $k$ such that there exists an open cover of $X$, $U_0,\ldots,U_k$ with the property that the inclusion of each $U_j$ in $X$ is nullhomotopic.
The only facts that we shall need are the following:
\begin{enumerate}
 \item $\cat(X) \leq \dim(X)$;
 \item $\e_0(X) \leq \cat(X)$;
 \item if $\overline X \to X$ is a covering map, then $\cat(\overline X) \leq \cat(X)$.
\end{enumerate}
These standard facts may be found in \cite{CLOT}. Now we can state a simple consequence of \propref{prop:cat}.

\begin{theorem}\label{cat:theorem}
Let $M$ be a compact manifold with almost nonnegative sectional curvature and let $\overline{M}\to M$ be a finite covering with associated ANSC-bundle
$F\to\overline{M}\xrightarrow{p} N$. Then
\[
\cat(M)\geq \e_0(\widetilde{M})+\dim(N)
\]
where $\widetilde M$ is the universal cover of $M$.
\end{theorem}

\begin{proof}
By Theorem \ref{thm:replace} $\overline M$ is nilpotent and $p$ is quasi-nilpotent, so we can then apply Proposition \ref{prop:cat} to obtain
\[
\e_0(\overline{M})\geq \e_0(F)+\dim(N)\,.
\]
By the properties of category listed above and the fact that $F \to \overline M \to N$ is equivalent to the universal cover fibration with $F \simeq \widetilde M$, we have
\[
\cat(M)\geq\cat(\overline{M})\geq \e_0(\overline{M})\geq \e_0(F)+\dim(N) = \e_0(\widetilde M) + \dim(N)\,.
\]
\end{proof}
Although \thmref{cat:theorem} follows immediately from \propref{prop:cat}, it is a powerful constraint on the structure of $M$. Here is a more
intrinsic form of the inequality under an extra hypothesis.

\begin{corollary}
Let $M$ be a compact manifold with almost nonnegative sectional curvature and assume that $\pi_1(M)$ is torsion-free. Then the
cohomological dimension $\cd(\pi_1(M))$ of $\pi_1(M)$ is finite, and
\[
\cat(M)\geq \e_0(\widetilde{M})+\cd(\pi_1(M))\,.
\]
\end{corollary}
\begin{proof}
If a torsion-free group $G$ has a finite index subgroup $H$ of finite cohomological dimension, then, by a result of Serre \cite[Chapter VIII, Theorem 3.1]{Brown},
$G$ itself has finite cohomological dimension, and $\cd(G)=\cd(H)$. By \cite[Theorem C]{KPT} $\pi_1(M)$ has a finite index subgroup which is finitely generated,
torsion-free and nilpotent. The cohomological dimension of such a group is precisely the dimension of the associated nilmanifold, and the conclusion then
follows from Theorem \ref{cat:theorem}.
\end{proof}
Recall that Bochner's Theorem states that a compact manifold $M$ with non-negative Ricci curvature obeys a Betti number condition: $\b_1(M) \leq \dim(M)$.
This type of inequality was refined in \cite{Op} (also see \cite{OS}) with $\dim(M)$ being replaced by $\cat(M)$. Here we have sharper information which is a topological analogue of Yamaguchi's theorem \cite{Yam}.

\begin{corollary}\label{cor:bcat}
Let $M$ be a compact manifold with almost nonnegative sectional curvature and assume that $\b_1(M)=\cat(M)$. Then $M$ is homeomorphic to the torus $T^{\b_1(M)}$.
\end{corollary}
\begin{proof}
Set $\b_1(M)=n$ and let $F\to\overline{M}\to N$ be the fiber bundle structure of the finite covering $\overline{M}\to M$. Notice that
$\b_1(N)=\b_1(\overline M) \geq \b_1(M)=n$ since $F$ is simply connected and $\overline M$ finitely covers $M$.
For a nilmanifold one always has $\dim(N)\geq \b_1(N)$ by the discussion within
\examref{exam:nilman}, so by Theorem \ref{cat:theorem},
\[
\cat(M)\geq \e_0 ({\widetilde{M}})+\dim(N)\geq \e_0(\widetilde{M})+\b_1(N) \geq \e_0(\widetilde{M})+\b_1(M)
\]
The hypothesis then implies that $\e_0(\widetilde{M})=0$. But this means that $\widetilde{M}$ is contractible since $\widetilde M \simeq F$ and $F$ is an
orientable closed manifold with non-zero top class in $H^*(F;\Q)$. If $\widetilde M$ is contractible, then $M$ is a $K(\pi,1)$. Moreover, the above equalities also imply that $\b_1(N)=\dim(N)$. By \lemref{lem:rhtnil}, this can only happen if $N$ is diffeomorphic to a torus $T^{\b_1(N)}$. We then have a finite covering $T^{\b_1(N)}\to M$ which gives an injection $\bZ^n\hookrightarrow\pi$, hence $n=\b_1(N)=\b_1(M)=\b_1(\pi)$. By lemma \ref{lem:group} below, $\pi\cong\bZ^n$, hence $M$ is homeomorphic to a torus.
\end{proof}

\begin{remark}\label{rem:diffeoM}
If we somehow knew that $M$ was a nilmanifold above, then we could infer that $M$ must be diffeomorphic to a torus. Alternatively, we could follow
Yamaguchi's rather complicated surgery approach in \cite{Yam} to obtain the stronger result.
\end{remark}

\begin{lemma}\label{lem:group}
If $\Gamma \cong \bZ^m$ is a finite index subgroup of a torsionfree
group $\pi$ and $\b_1(\pi)=m$, then $\pi\cong \bZ^m$.
\end{lemma}

\begin{proof}
Note first that the transfer map for finite coverings implies that
$H_*(\Gamma;\bQ) \to H_*(\pi;\bQ)$ is surjective. In particular, we have
a surjection on rationalized abelianizations,
\[
\Gamma_{\rm ab} \otimes \QQ = H_1(\Gamma;\QQ) \to H_1(\pi;\QQ) =
\pi_{\rm ab} \otimes \QQ\,.
\]
But $\b_1(\Gamma)=\b_1(\pi)$, and a surjection of rational vector spaces
of the same dimension is an isomorphism, so $\QQ^m \cong \Gamma_{\rm
ab} \otimes \QQ \cong \pi_{\rm ab} \otimes \QQ$. We have the following
commutative diagram.
\[\xymatrix{
\Gamma\cong \Z^m \ar[r]^-i \ar[d]_-\cong & \pi \ar[d]^-p \\
\Gamma_{\rm ab}\cong \Z^m \ar[r]^-{i_{\rm ab}} \ar[d]_-{\otimes \QQ} &
\pi_{\rm ab} \ar[d]^-{\otimes \QQ} \\
\QQ^m \ar[r]^-\cong & \QQ^m
}
\]
Note that, because the bottom row is an isomorphism, $i_{\rm
ab}$ is an
injection. We claim that ${\rm Ker}(p)=0$, so $p$ is an
isomorphism (since it is a surjection by definition). Suppose $x \in \pi$ and $p(x)=0$.
Now, $\Gamma$ has finite index in $\pi$ and if $x^s\Gamma = x^t\Gamma$ (for $s > t$
say), then $x^{s-t} \in \Gamma$, so there exists some $r \in \bN$ such that
$x^r \in \Gamma$. But then we have the contradiction
\[
0 \not = i_{\rm ab}(x^r) = p(i(x^r)) = 0\,.
\]
Therefore, $x^r=e$, where $e$ is the identity of $\pi$. But $\pi$ is
torsionfree, so $r=0$ and $x=e$. Hence $p$ is injective and $p\colon\pi \to
\pi_{\rm ab}$ is an isomorphism. Therefore, $\pi$ is a finitely generated
torsionfree abelian group; hence $\pi\cong \Z^m$ (since $\b_1(\pi)=m$).
\end{proof}

Here is another Bochner-like result about the fundamental group that uses a weaker hypothesis than \corref{cor:bcat}.

\begin{corollary}\label{cor:e0bcat}
Let $M$ be a compact manifold with almost nonnegative sectional curvature and assume that $\b_1(M) + \e_0(\widetilde M)=\cat(M)$.
Then in the associated ANSC bundle $F \to \overline M \to N$, we have:
\begin{enumerate}
 \item $N$ is a torus of rank $\b_1(M)$ and
 \item $\pi_1(M)$ is free abelian.
\end{enumerate}
\end{corollary}

\begin{proof}
Suppose $\b_1(M) + \e_0(\widetilde M)=\cat(M)$. From the string of inequalities
\[
\cat(M)\geq \e_0(\widetilde{M})+\dim(N)\geq \e_0(\widetilde{M})+\b_1(N) \geq \e_0(\widetilde{M})+\b_1(M)
\]
we again see that $\b_1(M)=\b_1(N)=\dim(N)$. Even though $\widetilde M$ may not be contractible, we can still conclude from
\lemref{lem:rhtnil} that $N$ is a torus with free abelian fundamental group of rank $\b_1(N)=\b_1(M)$. But then, by \lemref{lem:group},
we again infer that $\pi_1(M)$ is free abelian of the same rank.
\end{proof}

The results above can also be useful in computing the LS category of ANSC-manifolds. Consider the following example.
\begin{example}\label{exam:e0cat}
Take the Hopf bundle $S^3 \to S^7 \to S^4$ and mod out by the compatible principal $S^1$-actions on fibre and total space to
obtain a bundle $S^2 \to \CP^3 \to S^4$ with structure group $\mathrm{SO}(3)$ which preserves the round metric on $S^2$. The relative model for this bundle is given by
\[
(\Lambda(w_4,w_7),d_W) \to (\Lambda(w_4,w_7) \otimes \Lambda(v_2, v_3),D) \to (\Lambda(v_2,v_3),d_V)
\]
with non-zero differentials
\[
d_W(w_7)=w_4^2, \ d_V(v_3)=v_2^2, \ D(w_7)=w_4^2, \ D(v_3)=v_2^2 - w_4
\]
where the linear term $w_4$ in the last differential expresses the fact that the connecting homomorphism in the exact homotopy
sequence is non-trivial for $\pi_4(S^4) \to \pi_3(S^2)$. Let $KT$ denote the Kodaira-Thurston nilmanifold with minimal model
$(\Lambda(u_1,u_2,u_3,u_4),d)$ with $d(u_1)=d(u_2)=d(u_4)=0$ and $d(u_3)=u_1u_2$. (We choose this ordering to display
the fact that the Kodaira-Thurston manifold is a product of the $3$-dimensional Heisenberg nilmanifold and a circle.) Since $KT$ is $4$-dimensional,
there is a degree one map $\phi\colon KT \to S^4$ which on models is given by $\Phi\colon (\Lambda(w_4,w_7),d_W) \to (\Lambda(u_1,u_2,u_3,u_4),d)$;
$w_4 \mapsto u_1u_2u_3u_4$, $w_7 \mapsto 0$. Now pull back the bundle over $KT \to S^4$ to obtain a bundle
$S^2 \to X \to KT$ with model
\[
(\Lambda(u_1,u_2,u_3,u_4),d) \to (\Lambda(u_1,u_2,u_3,u_4) \otimes \Lambda(v_2,v_3),\tilde D) \to (\Lambda(v_2,v_3),d_V)
\]
with $\tilde D|_{u_i}=d$, $\tilde D(v_2)=0$ and $\tilde D(v_3)=v_2^2 - u_1u_2u_3u_4$. Note that by \cite{FY}, $X$ is ANSC.
The form of the model is determined by the pullback;
namely, the differential $\tilde D$ is given by taking $D$ and replacing all instances of generators $w_i$ by their images under $\Phi$.
It is then easy to see that the $6$-dimensional top class of $H^*(X;\Q)$ is represented by either $v_2^3$ or $v_2u_1u_2u_3u_4$ with
$\tilde D(v_2v_3)=v_2^3 - v_2u_1u_2u_3u_4$ identifying the classes in cohomology. These expressions then say that $\e_0(X)=5$.
Clearly we have $\e_0(KT)=4$ and $\e_0(S^2)=1$. so
\[
6=\dim(X) \geq \cat(X) \geq \e_0(X) = \e_0(KT) + \e_0(S^2) = 5\,.
\]
If $\cat(X)=\dim(X)$, then an old theorem of Berstein (see for instance \cite[Proposition 2.51]{CLOT}) says that, for some $\pi_1(X)$-module $\mathcal{A}$, there is an element $\alpha \in H^1(\pi_1(X);\mathcal{A}) \cong H^1(\pi_1(KT);\mathcal{A})$ with non-zero cup product $\alpha^6 \in H^6(\pi_1(KT);\otimes^6\mathcal{A})$. But this contradicts the fact that $\dim(KT)=4=\cd(\pi_1(KT))$ (where $\cd$ denotes cohomological dimension). Hence, we learn that $\cat(X)=5$.
\end{example}

\subsection{Non-negative Ricci curvature}\label{subsec:ricci}
In \cite{Op}, Bochner's estimate $\b_1(M) \leq \dim(M)$ for a non-negatively Ricci curved manifold $M$ was refined, using the following consequence of the Cheeger-Gromoll Splitting Theorem \cite{ChG}, to $\b_1(M)\leq \cat(M)$ .

\begin{theorem}[Cheeger-Gromoll Splitting]\label{cgthm}
If $M$ is a compact manifold with non-negative Ricci curvature,
then there is a finite cover $\overline M$ of $M$ with a
diffeomorphic splitting $\overline M \cong T^r \times F$. Further,
$F$ is simply connected and $T^r$ is flat.
\end{theorem}

In fact, while the new estimate was as stated above to mimic
Bochner, the proof showed that
\[
\b_1(M) +\cupl_{\Q}(F) \leq \cat(M)\,.
\]
Now we see however that the product splitting is an example of
an ANSC-type bundle -- namely, the trivial one,
$F \to \overline M \to T^r$. Therefore, the estimate of
\thmref{cat:theorem} holds in this situation as well.
Moreover, for rational coefficients, we have the general fact that
$\cupl_{\Q}(F) \leq \e_0(F)$, so we have the
following refinement.

\begin{theorem}\label{thm:ricci}
Suppose $M$ is a compact manifold with non-negative Ricci
curvature. Then
\[
\cat(M) \geq \b_1(M) + \e_0(\widetilde M)
\]
where $\b_1(M)$ is the first Betti number of $M$. Moreover, if $\pi_1(M)$ is torsionfree, then $\b_1(M)$ may be
replaced by $\cd(\pi_1(M))$.
\end{theorem}

Cohomogeneity one manifolds are a main source of
ANSC-manifolds (\cite{Tu}). These are the compact manifolds
equipped with
a compact Lie group action where the principal orbits are all of
codimension one. It is also true that every
cohomogeneity one manifold has a metric of non-negative Ricci
curvature, so in this case we see the connection.
However, we are unaware of a general result relating ANSC to
non-negative Ricci curvature. We also mention that in \cite{Cai}, a
Cheeger-Gromoll-type splitting is proven for compact manifolds with almost non-negative
Ricci curvature \emph{which also satisfy a certain lower
bound on injectivity radius}, so the estimate above holds in that
case as well. This should also be compared with the
result of Colding \cite{Col} that a compact manifold $M^n$ of
almost non-negative Ricci curvature with $\b_1(M)=n$ is homeomorphic to a torus $T^n$.

%%%%%%%%%%%%%%%%%%%%%%%%%%%%%%%%%%%%%%%%%%%%%%%%%%%%%%%%%%%%%%%%%%%%%%%%%%%%%%%%%%%%%
%                                                                                   %
%       THIRD REDUCTION: CONSTRAINTS ON THE ACTION AND FIBER                        %
%                                                                                   %
%%%%%%%%%%%%%%%%%%%%%%%%%%%%%%%%%%%%%%%%%%%%%%%%%%%%%%%%%%%%%%%%%%%%%%%%%%%%%%%%%%%%%

\section{The Third Reduction: Constraints on the Action and Fibre}\label{sec:fibre}
In \cite{Tu}, W. Tuschmann lists several interesting conjectures about ANSC-manifolds. Here we will see
that rational models have something to say in this regard. We first consider the following.

\begin{conjecture}[Conjecture 6.5 \cite{Tu}]\label{conj:action}
If $M$ is ANSC, then there exists a finite index subgroup of $\pi_1(M)$ that acts trivially on $\pi_2(M)$.
\end{conjecture}

The action of $\pi_1(M)$ on any $\pi_q(M)$ is detected by Whitehead products. Explicitly, if $\alpha\cdot\beta$
denotes the action of $\alpha\in \pi_1(M)$ on $\beta\in \pi_q(M)$, then
\[
 [\alpha,\beta] = \alpha\cdot\beta - \beta
\]
where $[\alpha,\beta]$ is the Whitehead product of $\alpha$ and $\beta$. The important thing for us is that
Whitehead products are detected in rational homotopy by the quadratic parts of the differential. Hence, at
least rationally, we can see the action in the differential of the model.

\subsection{Fibrations without action}\label{subsec:fibnoact}
Consider a fibration of ANSC-type $F \stackrel{i}{ \to} \overline M \stackrel{p}{\to} N=K(\pi,1)$.
Let's make several observations about the action of
$\pi_1(N)=\pi_1(\overline M)=\pi$ on $\pi_j(\overline M)=\pi_j(F)$ for $j \geq 2$. We know of course that
$F$ is simply connected, but its connectivity may be higher. Assume $F$ is $(k-1)$-connected. Therefore,
$\pi_j(\overline M)=0$ for $1< j \leq k-1$ and we have an exact sequence
\[
H_k(F) \cong \pi_k(F) \cong \pi_k(\overline M) \stackrel{h}{\to} H_k(\overline M) \to H_k(\pi) \to 0
\]
where $h$ is the Hurewicz homomorphism. For the last part, see \cite[Theorem 7.9]{Whi}. Also, if
$\alpha\in \pi$ and $\xi\in\pi_k(\overline M)$, then $h(\xi-\alpha\cdot\xi)=0$ since $\xi$ and
$\alpha\cdot\xi$ are freely homotopic. Thus we have the simple observation

\begin{lemma}\label{lem:hur}
If $h$ is injective, then the action of $\pi$ on $H_k(F)$ is trivial.
\end{lemma}

Of course, if $h\colon H_k(F) \to H_k(\overline M)$ is injective, then $\hat h\colon H^k(\overline M;\bQ) \to H^k(F;\bQ)$
is surjective. The following result relates the two consequences above directly.

\begin{theorem}\label{thm:ansc action}
Suppose that $F \stackrel{i}{ \to} \overline M \stackrel{p}{\to} N$ is an ANSC-type fibration with $(p^*)^{k+1} \colon H^{k+1}(N;\Q)
\to H^{k+1}(\overline M;\Q)$ injective and $\pi_j(F)=0$ for $j\leq k-1$. If the action
of $\pi_1(N)$ on $\pi_k(F)\otimes \Q = H^k(F;\bQ)$ is trivial, then $i^*\colon H^k(\overline M;\bQ) \to H^k(F;\bQ)$ is surjective.
Moreover, if $H^*(F;\bQ)$ is generated in degree $k$, then the fibration $F \stackrel{i}{ \to} \overline M \stackrel{p}{\to} N$ is totally
cohomologous to zero (TNCZ) and
\[
H^*(\overline M;\bQ) \cong H^*(N;\bQ) \otimes H^*(F;\bQ)
\]
as vector spaces.
\end{theorem}

\begin{proof}
Let a relative model for $F \to \overline M \to N$ be given by
$$(\Lambda V,d) \to (\Lambda V \otimes \Lambda W, D) \to (\Lambda W,\bar d).$$
Let $w \in W^k \cong \pi_k(F)\otimes \Q=H^k(F;\bQ)$ be a $\bar d$-cocycle. Then, because $\pi_1(N)$
acts trivially on $H^k(F;\bQ)$, we have
\[
Dw = \sum_I c_{I}v_I \eqqcolon \tau
\]
where $|I|=k+1$ and the $v_I$ are products of $k+1$ generators of $V^1$. (Note that the triviality of the action
is displayed by the differential having no quadratic terms.)
The facts that $Dx=dx$ for any $x\in \Lambda V$
and $D^2=0$ then imply that $\tau$ is a $d$-cocycle. If $\tau$ were not a coboundary,
then this would contradict the injectivity of $H^{k+1}(N;\bQ)$ in $H^{k+1}(\overline M;\bQ)$,
so we must have $d\sigma=\tau$ for some $\sigma \in (\Lambda V)^k$. But then the
element $w-\sigma$ is a $D$-cocycle that maps to $w \in \Lambda W$. Since, $w$ was
an arbitrary element of $H^k(F;\bQ)$, we see that $i^*$ is surjective.
\end{proof}

\begin{remark}
We have stated the result above for ANSC-type fibrations because that is our focus.
An examination of either proof, however, shows that the base being a nilmanifold
plays no role. Therefore, $N$ may be replaced above by any base space $B$ such that
$\pi_1(B)$ acts trivially on $H^*(F;\bQ)$ and $p^*$ is injective. If, for instance, $\chi(F) \not = 0$, then the
transfer of \cite{CG} shows that $p^*$ \emph{is} injective, so this is one case where the injectivity hypothesis holds.
\end{remark}

\begin{example}\label{exam:GT}
The flag manifold $G/T$, where $T$ is the maximal torus of the compact Lie group $G$,
has non-negative sectional curvature and $H^*(G/T;\bQ)$ generated in degree $2$.
Therefore, if $G/T$ arises as the fibre $F$ in some ANSC-type fibration
$F \to \overline M \to N$, then a trivial action of $\pi_1(N)$ on $H^2(G/T;\bQ)$
is detected in cohomology by the fibration being TNCZ.
\end{example}

\subsection{Fibrations with action}\label{subsec:fibwithact}
A mix of classical and rational homotopy theory can sometimes identify non-trivial actions. The
following result is an example of this in the context of ANSC manifolds.

\begin{proposition}[\!\! Compare {\cite[Proposition 4.100]{FOT}}]\label{prop:nontrivaction}
Let $F \stackrel{i}{ \to} \overline M \stackrel{p}{\to} N$ be an ANSC-type fibration with $\pi_j(F)=0$ for $j\leq k-1$.
Suppose that non-zero $a\in H^1(N;\bQ)$ and $\om\in H^k(F;\bQ)$ obey $p^*a \cup \tilde\om = 0$ where
$i^*(\tilde\om)=\om$. Then the action of $\pi_1(N)$ on $\pi_k(F)=H_k(F)$ is nontrivial.
\end{proposition}

\begin{proof} First, we can take multiples of $a$ and $\om$ so that they are integral. We therefore assume this. Now
note that the condition $i^*(\tilde\om) = \om$ is equivalent to saying that $\tilde\om|_{\mathrm{Im}(h)} \not = 0$, where
$\tilde\om \in H^k(\overline M) \to \Hom(H_k(\overline M;\bZ),\bZ)$ is considered dual to homology and operating on the image of
Hurewicz in $H_k(\overline M;\bZ)$. So now take $\gamma \in \pi_k(F)=H_k(F)$ such that $\tilde\om(h(\gamma))\neq$
and $\alpha \in \pi_1(N)=\pi_1(\overline M)$ such that $p^*a(h(\alpha))\neq 0$. So we are now thinking
of $\alpha\in\pi_1(\overline M)$ and $\gamma\in\pi_k(\overline M)$.

As we mentioned in Subsection \ref{subsec:fibnoact} the deviation of the action $\alpha \cdot \gamma$ from being
trivial is detected by the Whitehead product.
Thus, to show that the action of $\pi_1(N)=\pi_1(\overline M)$ on $\pi_k(\overline M)=H_k(F)$ is nontrivial,
it is sufficient to show that the Whitehead product $[\alpha, \gamma]$ is nonzero.
The cohomology classes $p^*a$ and $\tilde\om$ give a map
\[
p^*a \times \tilde\om \colon \overline M \to K(\bZ,1)\times K(\bZ,k)
\]
which, composed with $\iota_1\cup \iota_2 \colon K(\bZ,1)\times K(\bZ,k) \to K(\bZ,k+1)$ yields
\[
(p^*a \times \tilde\om)^*(\iota_1 \cup \iota_2)^*(\iota_3)= p^*a \cup \tilde\om = 0\,.
\]
Here, $\iota_j$ is the fundamental cohomology class of $K(\bZ,j)$. The equality $p^*a \cup \tilde\om = 0$
then shows that there is a lifting $\phi$ in the following diagram (where the right square is a pullback).
\[
\xymatrix{
& E \ar[d] \ar[r] & PK(\bZ,k+1) \ar[d] & \\
\overline M \ar[r]_-{p^*a \times \tilde\om} \ar[ur]^-\phi & K(\bZ,1)\times K(\bZ,2)
\ar[r]^-{\iota_1 \cup \iota_2} & K(\bZ,k+1)    }
\]

Now, the minimal model of $E$ is apparent: $\mathcal{M}_E =(\Lambda(x,y,z),d)$ with $|x|=1$, $|y|=k$, $|z|=k$ and the only
non-zero differential is $dz=xy$. The quadratic part of the differential, $d_1$, corresponds to the Whitehead product, up to sign so
we see that the Whitehead product in $\pi_k(E)\otimes \bQ$ is non-zero, $[\hat\iota_1,\hat\iota_2]\otimes \bQ \not = 0$, where $\hat\iota_1 \in \pi_1(K(\bZ,1))$ and $\hat\iota_2 \in \pi_2(K(\bZ,\textcolor{red}{2}))$ are the generators of the respective homotopy groups. But then any
integral multiple of $[\hat\iota_1,\hat\iota_2]$ is also non-zero as well. Now, since $p^*a(h(\alpha))\neq 0$ and $\tilde\om(h(\gamma))\neq 0$,
the lift $\phi$ can be used to push the Whitehead product $[\alpha, \gamma] \in \pi_k(M)$ forward to an integral
multiple of $[\hat\iota_1,\hat\iota_2]$. Since the latter is nontrivial, so is the former and we are done.
\end{proof}

\begin{example}
Suppose that $F \to M \to S^1$ has a model of the form
\[
\Lambda(x) \otimes \Lambda(v_1, v_2, w_1, w_2, w_3)\,,
\]
in which $|x| = 1$, $|v_i| = 2$, and $|w_j| = 3$, with differentials
\[
D(x) = 0, \quad D(v_1) = 0, \quad D(v_2) = xv_1, \quad D(w_1) = v_1^2, \quad D(w_2) = v_2^2 + 2 x w_2
\]
and
\[
D(w_3) = v_1 v_2 + x w_1\,.
\]
Then the action of $\pi_1(M) = \bZ$ on $\pi_2(M)$ -- as reflected in the differential $D(v_2) = xv_1$ -- is nilpotent but not trivial. We also see this from \propref{prop:nontrivaction} since $x \cup v_1=0$ (in cohomology) and $v_1$ is also a non-zero class
in $H^*(F;\bQ)$. Note that we do not have a surjection $H^*(M;\bQ) \to H^*(F;\bQ)$ since $v_2$ gives a non-zero class in $H^*(F;\bQ)$.

This example can be made ``geometric'' by realizing the model sequence
\[
(\Lambda(x),d=0) \to (\Lambda(x,v_1, v_2, w_1, w_2, w_3),D) \to (\Lambda(v_1, v_2, w_1, w_3,w_2),d)
\]
as a fibration $F \to M \to S^1$, then adjusting the map $M \to S^1$ to be smooth and finally adjusting the pullback of the $1$-form
on $S^1$ to get a submersion. This results in a mapping torus fibre bundle with the same rational homotopy characteristics as
the model sequence. Now, rationally, the model of $F$ may be displayed as a KS-extension
\[
(\Lambda(v_1,v_2,w_1,w_2),d) \to (\Lambda(v_1, v_2, w_1,w_2, w_3),\bar d)\otimes (\Lambda(w_3),d=0)\,,
\]
with non-trivial differentials $\bar d w_1=dw_1=v_1^2$, $\bar d w_2 =dw_2 = v_2^2$ and $\bar d w_3=v_1v_2$. Thus,
$(\Lambda(v_1, v_2, w_1, w_2),d)$ a model for $S^2 \times S^2$ and, up to rational homotopy, $F$ may be viewed as the
principal $S^3$-bundle over $S^2 \times S^2$ induced from the Hopf bundle $S^3 \to S^7 \to S^4$ by the top class map $S^2 \times S^2 \to S^4$. Note that this is the simplest simply connected non-formal manifold. Also note that the total space $M$ and fibre $F$ are elliptic spaces. Finally note that we have $p^*$ injective even though $\chi(F)=0$. (A general explanation for this is that a more general type of transfer exists. For a fibration $F \to E \to B$ with compact fibre $F$, if there exists a map $f\colon E \to E$ over $B$ and the Lefschetz number $L(\bar f)\neq 0$, where $\bar f\colon F \to F$ is the restriction of $f$, then a transfer map $\tau\colon H^*(M;\bQ) \to H^*(F;\bQ)$ is defined
such that $\tau p^*$ is multiplication by $L(\bar f)$. Here, we can define a self-map of $M$ at the model level by
$\phi(x)=0,\ \phi(v_j)=a v_j,\ \phi(w_j)=a^2 w_j$. The Lefschetz number of the restricted map on $F$ is easily seen to be
$L(\bar f) = 1+2a -2a^3 -a^4$ and choosing $a$ appropriately gives something non-zero.)
\end{example}

There is a whole family of examples like this one, obtained by adjusting the number and dimensions of the even spheres from which
$F$ is constructed, as well as by ``truncating'' with higher-degree odd generators. For instance, we have the following variation.

\begin{example}
Suppose that $F \to M \to S^1$ has a model
\[
\Lambda(x) \otimes \Lambda(v_1, v_2, w_1, w_2, w_3, w_4)\,,
\]
in which $|x| = 1$, $|v_i| = 2$, and $|w_j| = 5$, with differentials
\[
D(x) = 0 = D(v_1), \quad D(v_2) = xv_1, \quad D(w_1) = v_1^3, \quad D(w_2) = v_1^2 v_2 + x w_1\,,
\]
\[
D(w_3) = v_1v_2^2 + 2 x w_2, \quad D(w_4) = v_2^3 + 3 x w_3\,.
\]
Again, here, we have that $M$ is elliptic and $\chi(F) = 0$. There is an evident pattern present in these examples.
\end{example}

These examples confirm the phenomenon that a non-trivial action in degree-two entails a non-trivial action on (or, more generally,
non-trivial involvement of $x$ in differentials of) higher-degrees if $M$ is finite dimensional. More generally, a non-trivial action in any even degree will entail similar constraints.

\subsection{The case in which \texorpdfstring{$F$}{Lg} is elliptic with positive Euler characteristic}\label{subsec:F0}
Long ago, R. Bott conjectured that closed manifolds with non-negative sectional curvature are elliptic (see \cite{FOT} for example).
This was extended in \cite[Conjecture 6.1]{Tu} to manifolds which are almost non-negatively curved in the generalized sense.
Also, Hopf conjectured that a positively curved manifold has positive Euler characteristic. If these conjectures are true, then
it is possible that in an ANSC-bundle $F \to M \to N$, the fibre could be an $F_0$-space (see Subsection \ref{subsec:min}).
(In \cite{AK}, the authors study such manifolds as \emph{total} spaces of fibrations.)

We will prove two results with this hypothesis on the fibre.

\begin{theorem}\label{thm:zero derivs = trivial}
Suppose that $F \stackrel{i}{ \to} \overline M \stackrel{p}{\to} S^1$ is an ANSC-type fibration with $F$ an $F_0$-space. If there are no non-trivial degree-zero, nilpotent derivations of
$H^*(F; \bQ)$, then the fibration is rationally trivial: $\overline M$ and $S^1 \times F$ are of the same rational homotopy type.
\end{theorem}

\begin{proof}
Write a Sullivan model for the fibration as
\[
\Lambda(v) \to (\Lambda(v) \otimes\Lambda(W), D) \to (\Lambda W, d_W)\,.
\]
We will show that $(\Lambda(v) \otimes\Lambda(W), D)$, which is the minimal model of $M$, is isomorphic to $(\Lambda(v) \otimes\Lambda(W), 1\otimes d_W)$, the minimal model of
$S^1 \times F$, from which the assertion follows.

To this end, write the differential $D$ (following the general form of \examref{exam:nilrel} or \eqref{differential_D}) as
\[
D\chi= d_W\chi + v\theta(\chi)\,,
\]
for $\chi \in \Lambda W$, to define $\theta$, a degree-zero derivation of $\Lambda W$ that satisfies $d_W \theta = \theta d_W$ and thus passes to cohomology. Since $F$ is an $F_0$-space, its minimal model has the form $W = \langle x_1, \ldots, x_r \rangle \oplus \langle y_1, \ldots, y_r \rangle$, with $d_W(x_i) = 0$ and $d_W(y_j) \in \Lambda(x_1, \ldots, x_r)$. Furthermore, the cohomology of $\Lambda W$ is zero in odd degrees. Since each $x_i$ is a $d_W$-cycle, and by hypothesis $H^*(F;\bQ)$ has no non-trivial degree zero nilpotent derivations, we must have $\theta(x_i) = d_W(\eta_i)$ for some $\eta_i \in \Lambda W$ and for each $i$. So define a map $\phi \colon \Lambda(v) \otimes\Lambda(W) \to \Lambda(v) \otimes\Lambda(W)$ in the first place on generators as $\phi = \mathrm{id}$ on $\{v, y_1, \ldots, y_r\}$, and $\phi(x_i) = x_i + v \eta_i$ for each $i$, and then extending multiplicatively. Clearly $\phi$ is an isomorphism of algebras. If we define $D' = \phi^{-1}\circ D \circ \phi$, it is easy to check that $D'$ makes $(\Lambda(v) \otimes\Lambda(W), D')$ a Sullivan (minimal) model and we obtain an isomorphism of minimal models
\[
\phi \colon (\Lambda(v) \otimes\Lambda(W), D') \to (\Lambda(v) \otimes\Lambda(W), D)\,.
\]
For each $x_i$, we have
\[
D'(x_i) = \phi^{-1}\circ D \circ \phi(x_i) = \phi^{-1}\circ D(  x_i + v \eta_i) = \phi^{-1}\big(v\theta(x_i) - v d_W(\eta_i) \big)= 0\,.
\]
Now write
\[
D'\chi= d_W\chi + v\theta'(\chi)\,,
\]
just as before, to obtain $\theta'$, a degree-zero derivation of $\Lambda W$ that satisfies $d_W \theta' = \theta' d_W$. Notice that we have $\theta'(x_i) = 0$ for each $x_i$. On each $y_j$, we have $d_W \theta' (y_j) = \theta' d_W(y_j) = 0$, since $d_W(y_j) \in \Lambda(x_1, \ldots, x_r)$. Since $\theta' (y_j)$ is a $d_W$-cycle of odd degree, we must have
$\theta' (y_j) = d_W(\zeta_j)$ for some $\zeta_i \in \Lambda W$ and for each $j$.  Make a second change of basis as before, using
$\phi' = \mathrm{id}$ on $\{v, x_1, \ldots, x_r\}$, and $\phi'(y_j) = y_j + v \zeta_j$ for each $j$,  defining $D'' = (\phi')^{-1}\circ D' \circ \phi'$, and obtaining an isomorphism of minimal models
\[
\phi' \colon (\Lambda(v) \otimes\Lambda(W), D'') \to (\Lambda(v) \otimes\Lambda(W), D')\,.
\]
For each $x_i$, we have $D''(x_i) = (\phi')^{-1}\circ D' \circ \phi'(x_i) = (\phi')^{-1}\circ D' (x_i) = 0$, and
\[
\begin{aligned}
D''(y_j) &= (\phi')^{-1}\circ D' \circ \phi'(y_j) = (\phi')^{-1}\circ D'(  y_j + v \zeta_j) \\
&= (\phi')^{-1}\big(d_W(y_j) + v\theta'(y_j) - v d_W(\zeta_j) \big)= d_W(y_j),
\end{aligned}
\]
since $d_W(y_j) \in \Lambda(x_1, \ldots, x_r)$ and we have $\phi'(x_i) = x_i$, hence $(\phi')^{-1}(x_i) = x_i$, for each $i$.  Then the isomorphism
\[
\phi\circ \phi' \colon (\Lambda(v) \otimes\Lambda(W), 1\otimes d_W) \to (\Lambda(v) \otimes\Lambda(W), D)
\]
is the desired isomorphism of minimal models.
\end{proof}

The well-known \emph{Halperin conjecture} of rational homotopy (see Problem 1 of \cite[\S39]{FHT})
conjectures that any \emph{negative degree} derivation of the rational cohomology algebra of an $F_0$-space should be zero. We should point out, however, that where the Halperin conjecture has been established for a certain cohomology algebra $H$, it does not automatically follow that  degree-zero derivations also vanish on $H$.

\begin{proposition}\label{prop: degree-zero Halperin}
Suppose that $H$ is a Poincar{\'e} duality algebra with a single generator (monogenic) of even degree or two generators of even degree. Then there are no non-trivial degree-zero, nilpotent derivations of $H$.
\end{proposition}

\begin{proof}
 If $H$ is monogenic, e.g. we might have $H = H^*(\CP^n;\bQ)$ with the generator of degree $2$, then write the generator as $x$.  Since $\theta$ is nilpotent, and $H^{|x|}$ is a rank one vector space, we must have $\theta = 0$ on  $H^{|x|}$. That is, we have $\theta = 0$ on $H$.

Now suppose that $H$ has two even-degree generators $x$ and $y$, with $|x| \leq |y|$. If $|y|$ is not a multiple of $|x|$, then $H$ is of rank one in degrees $|x|$ and $|y|$, and we have $\theta = 0$ for the same reason as in the monogenic case.  So suppose that we have $|y| = p |x|$ for some $p \geq 1$.  Then the only possibility for a non-zero nilpotent $\theta$ is $\theta(x) = 0$ and $\theta(y) = \lambda x^p$, if $x^p \not = 0 \in H$.  Now suppose $x^k$ is the highest power of $x$ that is non-zero in $H$ (so that we have $x^{k+1} = 0$).  Note that $k \geq p$.  Then, suppose that $\ell\geq 0$ is the maximum exponent for which $x^k y^\ell \not = 0$ but $x^k y^{\ell+1} = 0$.  Then $x^k y^\ell$ is a non-zero element of $H$ that annihilates $x$ and $y$, and hence must be a fundamental class of $H$ (of maximum degree). Now in this maximal non-zero degree, $H$ is a rank-one vector space, and thus $\theta = 0$ in the top degree. In particular, a fundamental class cannot be in the image of $\theta$.  However, we have $\theta( x^{k-1} y^{\ell+1}) = \lambda (\ell+1) x^k y^\ell$, and so we must have $\lambda = 0$.  That is, we have $\theta = 0$.
\end{proof}

So, for instance, if we have $F = S^{2m} \times S^{2n}$ for any $m, n$, then \thmref{thm:zero derivs = trivial} and \propref{prop: degree-zero Halperin} imply that any ANSC-type fibration $F \stackrel{i}{ \to} \overline M \stackrel{p}{\to} S^1$ is trivial and, in particular the action of the fundamental group on the higher homotopy groups of $\overline{M}$ is rationally trivial.  For more general bases, we have the following.

\begin{theorem}\label{thm:F0action}
Suppose $F$ is an $F_0$-space such that there are no non-trivial degree-zero, nilpotent derivations of
$H^*(F; \bQ)$. Then, for an ANSC-type fibration $F \to M \to N$,
the action of the fundamental group on the higher homotopy groups of $M$ is rationally trivial.
\end{theorem}

\begin{proof}
\thmref{thm:zero derivs = trivial} handles the case in which $N = S^1$.  But this case is sufficient since, for any element in $\pi_1(M)$ that acts
non-trivially, we may pull-back the fibration to obtain one over a circle that also would have non-trivial action.  From the point
of view of models, we may ``isolate" any generator in the minimal model of the base by projecting onto it, and then pushing-out to obtain a
relative model of some fibration over a circle.  Any occurrence of the ``isolated" generator in the quadratic part of a differential would be
preserved under this move.  So it is sufficient to rule out such in the special case of base equal to $S^1$.
\end{proof}

Notice that it is possible to have a non-zero nilpotent derivation on a cohomology algebra generated in degree $2$, so long as we move away from the $F_0$-space case.

\begin{example}
Let $X = (S^2 \times S^2) \# (S^2 \times S^2)$.  Then $H^*(X; \bQ)$ may be presented as
\[
H^*(X; \bQ) \cong \frac{ \bQ[a, b, \alpha, \beta]}{\big( a^2, b^2, \alpha^2, \beta^2, a\alpha, a\beta, b\alpha, b\beta, ab - \alpha \beta  \big)}\,,
\]
with $a, b, \alpha, \beta$ of degree $2$.  Define $\theta\colon H^*(X; \bQ) \to H^*(X; \bQ)$ as a linear map on generators by
\[
\theta(a) = -\alpha, \quad \theta(b) = 0, \quad \theta(\alpha) = 0, \quad \theta(\beta) = b\,.
\]
When extended as a derivation, this map preserves the ideal generated by the relations (note that we have $\theta(a\beta) = -\alpha\beta + ab$) and hence extends as a derivation of the cohomology algebra.  It is clearly nilpotent and non-trivial.  With more argument, one can see that this example corresponds, up to rational homotopy, to an ANSC-type fibration
\[
X \to \overline{M} \to S^1
\]
in which the action is non-trivial.
\end{example}

The following result was proven in \cite{LO}.

\begin{theorem}[Theorem 4.1 \cite{LO}]\label{thm:formtot}
Suppose $\CP^{n-1} \to E \to B^{2n}$ where $B^{2n}$ is a Poincar\'e duality space of formal dimension $2n$. Then $E$ is formal
if and only if $B$ is formal.
\end{theorem}
Also in \cite{LO}, quasi-nilpotent fibrations of the type $\CP^{n-1} \to E \to N^{2n}$ were constructed, where $N^{2n}$ is a $2n$-dimensional
symplectic nilmanifold. The theorem then showed that $E$ was symplectic, but non-K\"ahler. We have seen in Corollary \ref{cor:e0bcat} that
if $\e_0(\tilde M)+\b_1(M)=\cat(M)$, then in the associated ANSC-bundle $F \to \overline M \to N$, the nilmanifold $N$ is a torus
$T^{\b_1(M)}$. Since $\CP^{n-1}$ has positive curvature, it could be the fibre in an ANSC-bundle. (For the bundles of Theorem \ref{thm:formtot},
we would need to know that the structure group preserves the metric on $\CP^{n-1}$ in order to apply the Fukaya-Yamaguchi condition
\cite{FY} for ANSC.) We have the following.

\begin{corollary}\label{cor:app}
Suppose $\CP^{n-1} \to M \to N^{2n}$ is an ANSC-bundle. If $\e_0(\tilde M)+\b_1(M)=\cat(M)$, then:
\begin{enumerate}
 \item $N$ is a torus $T^{\b_1(M)}$;
 \item The action of $\pi_1(M)$ on $\pi_*(M)$ is rationally trivial;
 \item $M$ is formal.
\end{enumerate}
\end{corollary}

\begin{proof}
Suppose $\e_0(\tilde M)+\b_1(M)=\cat(M)$. Then by Corollary \ref{cor:e0bcat}, $N = T^{2n}$, a torus. But by Theorem \ref{thm:formtot} this means that $M$ is formal. By Theorem \ref{thm:F0action}, the action of $\pi_1(M)$ on $\pi_*(M)$ is rationally trivial.
\end{proof}

\section*{Acknowledgments}

The first author would like to thank Cleveland State University and the other authors for the warm hospitality received during the two stays in Cleveland, when this project was carried out. The first author was partially supported by a Juan de la Cierva - Incorporación Grant and by Project MTM2015-63612-P, both of Spanish {\em Ministerio de Econom\'ia y Competitividad}.

\end{document}